\documentclass[12pt]{amsart}
\usepackage{amsmath}
\usepackage{amssymb}
\newtheorem{them}{Theorem}[section]
\newtheorem{lem}[them]{Lemma}
\newtheorem{cor}[them]{Corollary}

\theoremstyle{definition}
\newtheorem{defn}[them]{Definition}
\newtheorem{Exam}[them]{Example}
\newtheorem{Que}[them]{Question}

\numberwithin{equation}{section}

\providecommand{\AMS}{$\mathcal{A}$\kern-.1667em%
\lower.25em\hbox{$\mathcal{M}$}\kern-.125em$\mathcal{S}$}
\begin{document}

\title[multiplier algebras]{Which multiplier algebras are $W^*$-algebras?}

\author[C. Akemann]{Charles A. Akemann}

\address{Department of Mathematics, University of California, Santa Barbara, CA 93106, USA}
\email{akemann@math.ucsb.edu}

\author[ M. Amini]{Massoud Amini}

\address{School of Mathematics, Tarbiat Modares University, Tehran 14115 134,
 Iran, and School of Mathematics, Institute for Research in Fundamental
Sciences (IPM), Tehran 19395-5746, Iran}
\email{mamini@modares.ac.ir}

\author[ M.B. Asadi]{Mohammad B. Asadi}

\address{School of Mathematics, Statistics and Computer Science,
 College of Science, University of Tehran, Enghelab Avenue, Tehran,
  Iran, and School of Mathematics, Institute for Research in Fundamental Sciences (IPM),
Tehran 19395-5746, Iran}

\email{mb.asadi@khayam.ut.ac.ir}

\thanks{The second and third authors were in part supported by a grant from IPM (91430215 \& 91470123)}

\begin{abstract}
We consider the question of when the multiplier algebra $M(\mathcal{A})$ of a $C^*$-algebra $\mathcal{A}$ is a $ W^*$-algebra, and  show that it holds for a stable $C^*$-algebra exactly when it is a $C^*$-algebra of compact operators. This implies that if for every Hilbert $C^*$-module $E$ over a $C^*$-algebra $\mathcal{A}$,
the algebra $B(E)$ of
adjointable operators on $E$ is a $ W^*$-algebra, then $\mathcal{A}$ is a $C^*$-algebra of compact
operators.

Also we show that a unital $C^*$-algebra $\mathcal{A}$ which is Morita equivalent to a $ W^*$-algebra must be a $ W^*$-algebra.
\end{abstract}
\maketitle
\section{Introduction}

The main theme of this paper is around the question of when the multiplier algebra $M(\mathcal{A})$ of a
$C^*$-algebra $\mathcal{A}$ is a $ W^*$-algebra? For separable $ C^*$-algebras, it holds exactly when $\mathcal{A}$ is a $C^*$-algebra of compact
operators \cite[Theorem 2.8]{Ake}. For general $C^*$-algebras, we get two partial results in this direction. First we give an affirmative answer for stable $C^*$-algebras and deduce that if for every Hilbert $C^*$-module $E$ over $\mathcal{A}$,
the algebra $B(E)$ of
adjointable operators on $E$ is a $ W^*$-algebra, then $\mathcal{A}$ is a $C^*$-algebra of compact
operators. This is related to our question (with a much stronger assumption) as for $E=A$ with its canonical Hilbert $\mathcal{A}$-module structure, $B(E)= M(\mathcal{A})$.
Second we show that if $M(\mathcal{A})$ is Morita equivalent to a $ W^*$-algebra, then it is a $ W^*$-algebra. This is also related to our question, as if $\mathcal{A}$ is a $C^*$-algebra of compact operators, then $M(\mathcal{A})$ is  a $ W^*$-algebra.

The two partial answers take into account the notions of Hilbert $C^*$-algebras and Morita equivalence which are somewhat historically related.
In 1953, Kaplansky introduces Hilbert $C^*$-modules to prove that
derivations of type I $AW^*$-algebras are inner. Twenty years
later, Hilbert $C^*$-modules appeared in the pioneering
work of Rieffel \cite{Rie}, where he employed them to study (strong) Morita equivalence of $C^*$-algebras.
Paschke studied Hilbert $C^*$-modules as a generalization of Hilbert
spaces \cite{Pas}.

Hilbert $C^*$-modules and Hilbert spaces differ in many aspects, such as existence of orthogonal complements for submodules (subspaces), self duality,
existence of orthogonal basis, adjointability of bounded operators, etc. However, when $\mathcal{A}$ is a $C^*$-algebra of compact
operators, then Hilbert
$\mathcal{A}$-modules behave like Hilbert spaces in having the above properties. Indeed these properties characterize $C^*$-algebras of compact
operators \cite{Aram,Fra,Mag,Sch}.

\section{$C^*$-algebras of compact operators}

In this section we give some characterizations of  $C^*$-algebras of compact operators using properties of multiplier algebras. We also show that these are characterized as $C^*$-algebras $\mathcal{A}$ for which
the algebra $B(E)$ of all adjointable
operators is a $W^*$-algebra, for any Hilbert $\mathcal{A}$-module $E$.

\begin{defn} A $C^*$-algebra $\mathcal{A}$ is called a $C^*$-algebra of compact operators if there exists a Hilbert
space $H$ and a (not necessarily surjective) $*$-isomorphism from $\mathcal{A}$ to $K(H)$, where $K(H)$ denotes the space of compact operators on $H$.
\end{defn}

This is exactly how Kaplansky characterized C*-algebras that were dual rings \cite[Theorem 2.1, p. 222]{K} (see also \cite{A}).

\begin{them} For a  $C^*$-algebra $\mathcal{A}$, the following are equivalent:

($i$) $\mathcal{A}$ is a $C^*$-algebra of compact operators.

($ii$) The strict topology on the unit ball of $M(\mathcal{A})$ is the same as the strong$^*$-topology (viewing $M(\mathcal{A})\subseteq \mathcal{A}^{**}$, the second dual of $\mathcal{A}$).
\end{them}

\begin{proof}
 Assume that ($i$) holds. Then $\mathcal{A}\cong c_0$-$\sum\bigoplus_\alpha K(H_\alpha)$. Let $a_\beta\to 0$ in the strict topology of the unit
ball of $M(\mathcal{A}) \cong \ell^\infty$-$\sum\bigoplus_\alpha B(H_\alpha)$. Without loss of generality, we may assume that $a_\beta\geq 0$, for all $\beta$. Let $\eta\in \bigoplus_\alpha H_\alpha$ be a unit vector with $\eta_\alpha=0$ except for finitely many $\alpha$. Let $p_\alpha$ be the rank
one projection onto the non-zero $\eta_\alpha$ and $p_\alpha = 0$, otherwise. Then  $p =\sum p_\alpha\in\mathcal{A}$, thus $\|a_\beta p\|\to  0$. Therefore $\|a_\beta\eta\|\to 0$, and the same holds for any $\eta$ in the unit ball of $\bigoplus_\alpha H_\alpha$, as $\{a_\beta\}$ is norm bounded. Hence $a_\beta\to 0$ in the strong$^*$ topology.

Conversely if $a_\beta\geq  0$ and $a_\beta\to 0$ in the strong$^*$ topology. As above, for any rank one projection
$p\in \mathcal{A}$, $\|a_\beta p\|=\|pa_\beta\|\to 0$. Thus $p$ can be replaced by any finite linear combination of such minimal
projections, and this set is dense in $\mathcal{A}$. Since $\{a_\beta\}$ is norm bounded, $a_\beta\to 0$ in the strict
topology. This shows that ($i$) implies ($ii$).

Now assume that ($ii$) holds. By \cite[Theorem 2.8]{Ake}, we need only to prove that $M(\mathcal{A}) = \mathcal{A}^{**}$.
For any positive element $b$ in the unit ball of $\mathcal{A}^{**}$, there is a net $\{a_\beta\}$ in the unit ball of $\mathcal{A}$ that converges to $b$ in strong$^*$ topology. Thus the net is strong$^*$ Cauchy, and hence convergent in the strict topology to an element
of $M(\mathcal{A})$, as $M(\mathcal{A})$ is the completion of $\mathcal{A}$ in the strict topology \cite[Theorem 3.6]{B}. Therefore $b\in M(\mathcal{A})$, and we are done.
\end{proof}

Another characterization of $C^*$-algebras of compact operators could be obtained as a non unital version of the following result of J.A.
Mingo in \cite{Min}, where he investigates the
multipliers of stable $C^*$-algebras.

\begin{lem}
Suppose that $H$ is a separable infinite dimensional Hilbert space and
$\mathcal{A}$ is a unital $C^*$-algebra such that the multiplier
algebra $M(\mathcal{A} \otimes K(H))$ is a $W^*$-algebra. Then
$\mathcal{A}$ is a finite dimensional $C^*$-algebra.
\end{lem}

We recall that a projection $p$ in a $C^*$-algebra $\mathcal{A}$
is called finite dimensional if $p \mathcal{A} p$ is a finite
dimensional $C^*$-algebra. To prove a non unital version of Mingo's result, we need some lemmas. The first lemma is well-known, see for instance \cite[Corollary 1.2.37]{AM}.

\begin{lem}
If $\mathcal{A}$ is a  $C^*$-algebra and $p$ is a projection in the multiplier
algebra $M(\mathcal{A})$, then $M(p\mathcal{A}p)\cong pM(\mathcal{A})p$, as $C^*$-algebras.
\end{lem}

\begin{lem}
Let $H$ be a Hilbert space and $\mathcal{A}$ be a  $C^*$-algebra. If $\mathcal{A}\otimes K(H)$ is  $C^*$-algebra of compact operators, then so is $\mathcal{A}$.
\end{lem}

\begin{proof}
Suppose not. Then there is an element $b \in\mathcal{A}^{+}$ such that the spectral projection $\xi_1(b)$ of
$b$ corresponding to $\{1\}$ is not finite dimensional in $\mathcal{A}$. Let $q$ be a one-dimensional projection in $K(H)$.
Then $(b\otimes q)^n$ is a decreasing sequence in the unit ball of the C*-algebra $\mathcal{A}\otimes K(H)$ of compact operators. By Theorem 2.2 it converges strictly, hence (because it is decreasing) in norm to $\xi_1(b)\otimes q \in \mathcal{A}$. Because $\mathcal{A}\otimes K(H)$ is a $C^*$-algebra of compact operators, the projection   $\xi_1(b)\otimes q$
 must be finite rank, but
$$(\xi_1(b)\otimes q)(\mathcal{A}\otimes K(H))(\xi_1(b)\otimes q)=\xi_1(b)\mathcal{A}\xi_1(b)\otimes qK(H))q,$$ and the dimension of  $\xi_1(b)\mathcal{A}\xi_1(b)$ is
not finite by our assumption about $b$.
\end{proof}

The next theorem is known for separable $C^*$-algebras \cite{Ake}, here we prove it with separability replaced by stability.

\begin{them}
If $\mathcal{A}$ is a stable $C^*$-algebra such that the multiplier
algebra $M(\mathcal{A})$ is a $W^*$-algebra, then $\mathcal{A}$ is a $C^*$-algebra of compact operators.
\end{them}
\begin{proof}

In order for the $C^*$-algebra $\mathcal{A}$ to be a
$C^*$-algebra of compact operators, it is necessary and
sufficient that every positive element in $\mathcal{A}$ can be
approximated by a finite linear combination of finite dimensional
projections. Let $a$ be a positive element in $\mathcal{A}$ and $0
\leq a \leq 1$. Since the multiplier algebra $M(\mathcal{A})$
is a $W^*$-algebra, we can define $p \in M(\mathcal{A})$ as the
spectral projection of $a$, corresponding to an interval of the form
$[s , t]$ where $0 < s < t$. It suffices to show that $p
\mathcal{A} p$ is finite dimensional. Let $g: [0,1] \rightarrow
[0,1]$ be a continuous function vanishing at $0$, such that $g(r)
= 1$ for all $r \in [s , t]$. Then $g(a) \in A$ and $g(a)p = p$.
Hence $p \in A$.

Now let $H$ be a separable infinite dimensional Hilbert space. Since $\mathcal{A}$ is a stable $C^*$-algebra, $M(\mathcal{A})=M(\mathcal{A}\otimes K(H))$ is a $W^*$-algebra and by Lemma 2.4,
\begin{align*}M(p\mathcal{A}p\otimes K(H))&=M((p\otimes 1)(\mathcal{A}\otimes K(H))(p\otimes 1))\\&=(p\otimes 1)M(\mathcal{A}\otimes K(H))(p\otimes 1)\end{align*}
is a $W^*$-algebra. Therefore by Lemma 2.3, $p$ is finite rank.
\end{proof}

The non unital version of the Mingo's lemma follows.

\begin{cor}
Suppose that $H$ is a separable infinite dimensional Hilbert space and
 $\mathcal{A}$ is a $C^*$-algebra such that the multiplier
algebra $M(\mathcal{A}\otimes K(H))$ is a $W^*$-algebra, then $\mathcal{A}$ is a $C^*$-algebra of compact operators.
\end{cor}
\begin{proof}
Since $\mathcal{A}\otimes K(H)$ is stable, it is a $C^*$-algebra of compact operators,
 and so is $\mathcal{A}$ by Lemma 2.5.
\end{proof}

It is well known that if $\mathcal{A}$ is a $W^*$-algebra and $E$
is a selfdual Hilbert $\mathcal{A}$-module, then $B(E)$ is a
$W^*$-algebra. The converse is not true, as for $E=\mathcal{A}=
c_0$, $B(E)=\ell^\infty$ is a $W^*$-algebra \cite{Rie}. However,
if $\mathcal{A}$ is a $C^*$-algebra of compact operators on some
Hilbert space, then $B(E)$ is a $W^*$-algebra, for every Hilbert
$\mathcal{A}$-module $E$ \cite{Bak}. Here we show the converse.

Recall that the $C^*$-algebra $K(E)$ of compact operators on $E$
is generated by rank one operators $\theta_{\xi,\eta}(\zeta)=\xi
\langle \eta, \zeta \rangle$, for $\ \xi, \eta \in E$, and the
multiplier algebra $M(K(E))$ is isomorphic to $B(E)$. Also, if $H$
is a separable infinite dimensional Hilbert space, then $E=H
\otimes  \mathcal{A}$ is a Hilbert $C^*$-module over $\mathbb
C\otimes \mathcal A=\mathcal A$, denoted by $H_{\mathcal A}$.
It plays an
 important role in the theory of Hilbert $C^*$-modules.

\begin{them}
For any $C^*$-algebra $\mathcal{A}$, the following are equivalent:

($i$) $\mathcal{A}$ is a $C^*$-algebra of compact operators,

($ii$) $B(E)$ is a $W^*$-algebra, for each Hilbert
$\mathcal{A}$-module $E$,

 ($iii$) $B(H_{\mathcal A})$ is a $W^*$-algebra.
\end{them}
\begin{proof}
It is enough to show that ($iii$) implies ($i$). Since
$$K(H_{\mathcal A})=K(H \otimes  \mathcal{A}) \cong K(H) \otimes K( \mathcal{A})=
K(H) \otimes  \mathcal{A}  $$
 we have $B(H_{\mathcal A}) \cong M(K(H) \otimes  \mathcal{A} )$.
 By assumption, $B(H_{\mathcal A})$ is a
 $W^*$-algebra and so $\mathcal{A}$ is a $C^*$-algebra of compact
 operators by Corollary 2.7.

\end{proof}

J. Schweizer in \cite{Sch} remarked that for a $C^*$-algebra
$\mathcal{A}$, some problems on Hilbert $\mathcal{A}$-modules can
be reformulated as  problems on right ideals of $\mathcal{A}$,
since submodules of a full Hilbert $\mathcal{A}$-module are in a
bijective correspondence with the closed right ideals of
$\mathcal{A}$. Therefore, one may wonder if the previous
result could be reformulated in the language of right ideals. Actually,
if $I$ is a (closed) right ideal of $\mathcal{A}$, then $I$ is a
right Hilbert $\mathcal{A}$-module with inner product
$\langle a, b \rangle = a^* b$, for $a, b \in I$, and in this case, $K(E)$
equals to the hereditary $C^*$-algebra $I \cap I^*$ and so
$B(E)= M(I\cap I^*)$. Therefore, one may expect that  $C^*$-algebras
$\mathcal{A}$ of compact operators may be characterized by the
property that for  every hereditary $C^*$-subalgebra $\mathcal{B}$
of $\mathcal{A}$, $M(\mathcal{B})$ is a $ W^*$-algebra.

Unfortunately, this is not the case for non separable
$C^*$-algebras, as the following counterexample shows. However,
if $\mathcal{A}$ is separable and $p$ is a projection as in the
proof of Theorem 2.6, then $p \mathcal{A} p$ is a separable
$W^*$-algebra, hence finite dimensional (also see Theorem 2.8 in
\cite{Ake}).

\begin{Exam} For the Stone-Cech
compactification $\beta \mathbb{N}$ of the natural numbers, the algebra of continuous functions
$C(\beta\mathbb{N})$ is a $W^*$-algebra. Let $x$ be any point of $\beta
\mathbb{N}$ that is not a natural number and let $\mathcal{A}$ be
the $C^*$-subalgebra of $C(\beta \mathbb{N})$ consisting of those
functions vanishing at $x$. Let $\mathcal{B}$ be a hereditary C*-subalgebra
of $\mathcal{A}$ (which is an ideal, since $\mathcal{A}$ is abelian).
Then there is an open subset $U$ of $\beta \mathbb{N}$ such that
$\mathcal{B}$ consists of functions in $\mathcal{A}$ that
vanish outside $U$. Let $V$ be the closure of $U$. Then $V$ is
also open. For every $c \in C(V )$ we may extend $c$ by zero outside $V$, and thereby
view $C(V )$ as a $W^*$-subalgebra of
$C(\beta \mathbb{N})$. Observe that $M(\mathcal{B}) = C(V )$: clearly $\mathcal{B}$ is an ideal in $C(V )$, so it suffices to
note that for any $0 \neq c \in C(V )$, $c \mathcal{B} \neq 0$. To
see this, we note that $c$ is non-zero on a nonvoid open subset
$W$ of $V$, hence $W \cap U \setminus {x}$ is a nonvoid open set.
Hence there exists a non-zero continuous function $b$ with support
in $W \cap U \setminus {x}$. Thus $b \in B$ and $cb \neq 0$. Therefore
$M(\mathcal{B}) = C(V )$ is a $W^*$-algebra, but $\mathcal{A}$ cannot be a $C^*$-algebra of compact operators.
\end{Exam}

\section{Morita equivalence}

The notion of (strong) Morita equivalence of $C^*$-algebras
was introduced by M. Rieffel in \cite{Rie}. Two $C^*$-algebras $\mathcal{A}$ and
$\mathcal{B}$ are \textit{(strongly) Morita equivalent}
if there is an $\mathcal{A}$-$\mathcal{B}$-bimodule
$\mathcal{M}$, which is a left full Hilbert $C^*$-module over
$\mathcal{A}$, and a right full Hilbert $C^*$-module over
$\mathcal{B}$, such that the inner products
$_{\mathcal{A}}\langle \cdot, \cdot \rangle$ and $\langle \cdot,
\cdot \rangle_{\mathcal{B}}$ satisfy
$_{\mathcal{A}}\langle x, y \rangle z =x \langle y, z
\rangle_{\mathcal{B}}$ for all
 $x,y,z \in \mathcal{M}$.
Such a module $\mathcal{M}$ is called an
$\mathcal{A}$-$\mathcal{B}$-imprimitivity bimodule.

It would be interesting to investigate those properties of $C^*$-algebras which are preserved under
Morita equivalence. These include, among other things nuclearity, being type I, and simplicity \cite{Ara,Bee, Kus, Rae, Rae2, Z, Z2}.
Now if one of the two Morita equivalent $C^*$-algebras is a $W^*$-algebra,
it is natural to ask if so is the other. The answer to this question, as it posed is obviously negative, as Hilbert space $H$ is a
 $K(H)$-$\mathbb{C}$-imprimitivity bimodule, and so
 $C^*$-algebras $K(H)$ and $\mathbb{C}$ are  Morita equivalent. However we may rephrase that question in the following less trivial form.

\begin{Que}
Suppose that $C^*$-algebras $\mathcal{A}$ and $\mathcal{B}$ are
Morita equivalent and the $C^*$-algebra $M(\mathcal{A})$ is a
$W^*$-algebra, is it then true that $M(\mathcal{B})$ is a
$W^*$-algebra?
\end{Que}
By Theorem 2.8, we can show that the above property holds for
$C^*$-algebra $\mathcal{A}$ exactly when $\mathcal{A}$ is a
$C^*$-algebra of compact operators. In fact, we have the following
result.
\begin{them}
Let $\mathcal{A}$ be a $C^*$-algebra such that $M(\mathcal{B})$
is a $W^*$-algebra, for any $C^*$-algebra $\mathcal{B}$ which is
Morita equivalent to $\mathcal{A}$. Then $\mathcal{A}$ is a
$C^*$-algebra of compact operators.
\end{them}
\begin{proof}
Let $\mathcal{B}=K(H_\mathcal{A})$. Since $H_\mathcal{A}$ is a
full Hilbert $\mathcal{A}$-module, then $\mathcal{B}$ is Morita
equivalent to $\mathcal{A}$. By assumption,
$B(H_\mathcal{A})\cong M(\mathcal{B})$ is a $W^*$-algebra, hence
 $\mathcal{A}$ is a $C^*$-algebra of compact operators, by Theorem 2.8.
\end{proof}

However, we give an affirmative answer to the above question, when
both $C^*$-algebras are unital.

Recall that a Hilbert $C^*$-module $E$ on a $C^*$-algebra $\mathcal{A}$ is called \textit{self dual} if for
every bounded linear $\mathcal{A}$-module map $\varphi: E\rightarrow \mathcal{A}$
there is an element $y \in E$ such that $\varphi(\cdot) = \langle y, \cdot \rangle$.

\begin{lem}
Let $E$ be a right Hilbert $C^*$-module over a $C^*$-algebra
$\mathcal{A}$ such that $K(E)$ is unital. then\\
 $(i)$ $E$ is self dual.\\
 $(ii)$ $B(E)$ is a $W^*$-algebra, whenever $\mathcal{A}$ is a $W^*$-algebra.
\end{lem}
\begin{proof}
By hypothesis there are elements $x_1,\cdot\cdot\cdot,x_n$ and
$y_1,\cdot\cdot\cdot,y_n$ in $E$ such that $\sum_{i=1}^n
\theta_{x_i,y_i}=1 \in K(E)$. Thus, for every bounded linear
$\mathcal{A}$-module map
 $\varphi: E\rightarrow \mathcal{A}$ and $x \in E$ we have
 \begin{align*}
 \varphi(x)&=\varphi(\sum_{i=1}^n \theta_{x_i,y_i}x)= \varphi(\sum_{i=1}^n x_i \langle y_i, x\rangle)
 = \sum_{i=1}^n \varphi(x_i) \langle y_i, x \rangle
 \\& = \sum_{i=1}^n  \langle y_i \varphi(x_i)^*, x \rangle =
  \langle \sum_{i=1}^n y_i \varphi(x_i)^*, x \rangle.
  \end{align*}
  Therefore $\varphi(x)=\langle y, x \rangle$,
  where $y=\sum_{i=1}^n y_i \varphi(x_i)^*$. Hence $E$ is selfdual.

Now ($ii$) follows from ($i$) and  \cite[Proposition 3.10]{Pas}.
\end{proof}

Now  if $E$ is an
 $\mathcal{A}$-$\mathcal{B}$-imprimitivity bimodule, then $\mathcal{A}\cong
K_{\mathcal{B}}(E)$ and $\mathcal{B}\cong K_{\mathcal{A}}(E)$.
Therefore, the following partial answer to the above question follows from the above lemma.
\begin{them}
Suppose that unital $C^*$-algebras $\mathcal{A}$ and
$\mathcal{B}$ are Morita equivalent. Then $\mathcal{A}$ is a
$W^*$-algebra if and only if $\mathcal{B}$ is a $W^*$-algebra.
\end{them}

A similar result can be proved for operator algebras. Let
$\mathcal{A}$ and $\mathcal{B}$ be operator algebras. We say that
$\mathcal{A}$ and $\mathcal{B}$ are (strongly) Morita equivalent
if they are Morita equivalent in the sense of Blecher, Muhly,
Paulsen \cite{Ble}. In \cite{Ble}, it is proved that two
$C^*$-algebras are (strongly) Morita equivalent (as operator
algebras) if and only if they are Morita equivalent in the sense
of Rieffel.

\begin{them}
Suppose that unital operator algebras $\mathcal{A}$ and
$\mathcal{B}$ are Morita equivalent. Then $\mathcal{A}$ is a dual
operator algebra if and only if $\mathcal{B}$ is a dual operator
algebra.
\end{them}
\begin{proof}
Let $\pi :\mathcal{A}\rightarrow B(H)$ be a completely isometric
normal representation of $\mathcal{A}$ on some Hilbert space $H$.
Then there exist a completely isometric representation
 $\rho :\mathcal{B}\rightarrow B(K)$ of $\mathcal{B}$ on a Hilbert
 spaces $K$ and subspaces $X \subseteq B(K,H)$, $Y \subseteq B(H,K)$
 such that
  $$\pi (\mathcal{A})X \rho (\mathcal{B}) \subseteq X, \
 \rho (\mathcal{B})Y \pi (\mathcal{A})\subseteq Y, \
  \pi (\mathcal{A})=\overline{XY}^{\|\cdot \|}, \
   \rho (\mathcal{B}) = \overline{YX}^{\|\cdot \|}$$
   Since $\pi$ is normal, we have $\pi(\mathcal{A}) =
   \overline{\pi(\mathcal{A})}^{w*}$. Now $X \rho(\mathcal{B})Y \subseteq \pi(\mathcal{A})$
   implies that $X \overline{\rho(\mathcal{B})}^{w*}Y \subseteq
   \pi(\mathcal{A})$. Therefore
$$YX \overline{\rho(\mathcal{B})}^{w*}YX \subseteq
   Y \pi(\mathcal{A})X \subseteq \rho(\mathcal{B}),$$
   and so $\rho(\mathcal{B}) \overline{\rho(\mathcal{B})}^{w*} \rho(\mathcal{B}) \subseteq
   \rho(\mathcal{B})$.
   Since $\rho(\mathcal{B})$ is a unital algebra we have
$ \overline{\rho(\mathcal{B})}^{w*}  \subseteq \rho(\mathcal{B}),$
hence $ \overline{\rho(\mathcal{B})}^{w*} = \rho(\mathcal{B})$.
Therefore $\mathcal{B}$ is a dual operator algebra.

\end{proof}

\subsection*{Acknowledgment} The authors would like to thank Professor G. Eleftherakis, who suggested
 the proof of Theorem 3.5, and Professor J. Tomiyama for helpful comments.

\end{document}